\documentclass[runningheads]{llncs}

\usepackage[T1]{fontenc}

\usepackage{amsmath, amsfonts, mathtools, bm, amssymb}
\usepackage{graphicx}
\usepackage[svgnames]{xcolor}
\usepackage{booktabs}
\usepackage{array,multirow}
\usepackage{sidecap}
\usepackage[font=small,labelfont=bf]{caption}
\usepackage{cite}

\usepackage{tikz}
\usetikzlibrary{decorations.pathreplacing}

\usepackage{hyperref}

\spnewtheorem{algorithm}{Algorithm}{\bfseries}{}

\def\R			{\mathbb R}

\def\Sphere		{\mathbb{S}_1}
\def\Ball		{\mathbb{B}_2}
\def\I		    {\mathbb{I}}

\def\M	        {\mathcal M}
\def\Lebesgue	{\mathrm L}

\def\Radon		{\mathcal R}
\def\NRCDT		{\mathcal N}
\def\maxNRCDT	{\NRCDT_{\mathrm{m}}}

\def\GL			{\mathrm{GL}}

\def\d			{\mathop{}\!\mathrm{d}}

\def\bfx		{\mathbf{x}}
\def\bfy		{\mathbf{y}}

\def\bfzero		{\mathbf{0}}
\def\bfA		{\mathbf{A}}
\def\bfI		{\mathbf{I}}
\def\bftheta	{{\boldsymbol{\theta}}}

\def\P			{\mathcal P}

\def\Glue		{\mathcal I}

\def\F			{\mathbb F}
\def\G			{\mathbb G}

\DeclareMathOperator{\diam}{diam}
\DeclareMathOperator{\supp}{supp}

\DeclareMathOperator{\mean}{mean}

\DeclareMathOperator{\std}{std}

\DeclareMathOperator*{\argmin}{arg\,min}

\newcommand{\subsubset}{\subset\joinrel\subset}

\DeclareFontFamily{U}{mathx}{\hyphenchar\font45}
\DeclareFontShape{U}{mathx}{m}{n}{<-> mathx10}{}
\DeclareSymbolFont{mathx}{U}{mathx}{m}{n}
\DeclareMathAccent{\widebar}{0}{mathx}{"73}

\def\mNRCDT			{\textsubscript{m}NR-CDT}

\usepackage{color}

\begin{document}
\title{Max-Normalized Radon Cumulative Distribution Transform 
for Limited Data Classification}
\titlerunning{Max-Normalized Radon-CDT}

\author{%
Matthias Beckmann\inst{1}
\and
Robert Beinert\inst{2}
\and
Jonas Bresch\inst{2}
}
\authorrunning{M.~Beckmann et al.}

\institute{Center for Industrial Mathematics, University of Bremen,\\
Bibliothekstra{\ss}e 5, 28359 Bremen, Germany\\
\email{research@mbeckmann.de}
\and
Institut f\"ur Mathematik, Technische Universit\"at Berlin,\\
Stra{\ss}e des 17.\ Juni 136, 10623 Berlin, Germany\\
\email{\{beinert,bresch\}@math.tu-berlin.de}
}

\maketitle 
\begin{abstract}
The Radon cumulative distribution transform (R-CDT) exploits 
one-dimensional Wasserstein transport
and the Radon transform
to represent prominent features in images.
It is closely related to the sliced Wasserstein distance
and facilitates classification tasks,
especially in the small data regime, 
like the recognition of watermarks in filigranology.
Here, a typical issue is 
that the given data may be subject to affine transformations
caused by the measuring process.
The aim of this paper is to make the R-CDT 
and the related sliced Wasserstein distance
invariant under affine transformations.
For this,
we propose a two-step normalization of the R-CDT
and prove that our novel transform allows 
linear separation of affinely transformed image classes.
The theoretical results are supported by numerical experiments
showing a significant increase of the classification accuracy
compared to the original R-CDT.

.
\end{abstract}

\section{Introduction}

Automated pattern recognition and classification
play a central role in numerous applications and disciplines,
be it in medical imaging, biometrics, or document analysis.
Nowadays,
in the big data regime,
end-to-end deep neural networks provide the latest state of the art. 
In the small data regime,
however,
hand-crafted feature extractors and classifiers still stand their ground.
Ideally, 
the feature extractor is designed to transform different classes
to linear separable subsets.
This may,
for instance,
be achieved by 
the so-called Radon cumulative distribution transform (R-CDT)
introduced in~\cite{Kolouri2016},
which is based on one-dimensional optimal transport maps
that are generalized to two-dimensional data
by applying the Radon transform,
known from tomography~\cite{Ramm1996,Natterer2001}.
This approach shows great potential in many applications
\cite{Kolouri2017,DiazMartin2024,Shifat-E-Rabbi2021}
and is closely related to the sliced Wasserstein distance \cite{Bonneel2015,Shifat-E-Rabbi2023}.
A similar approach for data on the sphere is studied in \cite{Quellmalz2023,Quellmalz2024},
for multi-dimensional optimal transport maps in \cite{Moosmueller2023},
and for optimal Gromov--Wasserstein transport maps in \cite{Beier2022}.

A central inspiration for this paper is the 
application of pattern recognition techniques in filigranology%
---the study of watermarks.
These play a central role in dating historical manuscripts
as well as identifying scribes and papermills.
For automatic classification,
the main issue is the enormous number of classes
with only few members per class,
see WZIS%
\footnote{Wasserzeichen-Informationssystem: \url{www.wasserzeichen-online.de}.}.
An end-to-end processing pipeline for thermograms of watermarks 
including an R-CDT-based classification is proposed
in \cite{Hauser2024}, where the authors report classification invariance
with respect to translation and dilation of the watermark.
Other affine transformations caused, e.g., by unstandardized recording methods are,
however,
not included yet.

\paragraph{Contribution.}

The aim of this paper is to incorporate
invariance with respect to
affine transformations into the R-CDT,
this is,
to make the sliced Wasserstein distance
unaware of these transforms.
In difference to \cite{Shifat-E-Rabbi2021},
where the dataset is augmented to encode invariances,
we propose a two-step normalization of the R-CDT
for probability measures on $\R^2$.
To this end,
we first generalize the classical Radon transform to measures
in §~\ref{sec:radon} and, thereon, introduce
the novel \emph{max-normalized R-CDT} (\mNRCDT) in §~\ref{sec:ot-dist}.
Our main theoretical contribution is Theorem~\ref{thm:sep-max-nrcdt}
ensuring the linear separability of affinely transformed measure classes by \mNRCDT.
The theoretical findings are supported 
by proof-of-concept experiments 
in §~\ref{sec:num-ex}
showing a significant improvement of the classification accuracy 
by the proposed normalization,
especially in the small data regime.

\section{Radon Transform}
\label{sec:radon}

The main idea behind the classical Radon transform \cite{Natterer2001} is
to integrate a given bivariate function along all parallel lines
pointing in a certain direction.
This integral transform can also be interpreted
as projection of the given function
onto the line with orthogonal orientation.
In the following,
we briefly review the classical Radon transform for functions
and generalize the concept to measures.
Finally,
we study the effect of affine transformations 
on the Radon transform,
which is crucial
to solve the classification task at hand.

\subsection{Radon Transform of Functions}

Depending on $\bftheta \in \Sphere \coloneqq \{\bfx \in \R^2 \mid \lVert \bfx \rVert = 1\}$,
we introduce the \textit{slicing operator} $S_\bftheta\colon \R^2 \to \R$ by
\begin{equation*}
    S_\bftheta(\bfx) \coloneqq \langle \bfx, \bftheta \rangle,
    \quad
    \bfx \in \R^2.
\end{equation*}
Its preimages $S_\bftheta^{-1}(t)$, 
$t \in \R$,
are the lines $\ell_{t,\bftheta}$ 
in direction 
$\bftheta^\perp \coloneqq (\theta_2, - \theta_1)^\top \in \Sphere$ 
with distance $t$ to the origin.
More precisely,
we have 
\begin{equation*}
    \ell_{t,\bftheta} 
    \coloneqq 
    S_\bftheta^{-1}(t) 
    = 
    \{ t \bftheta + \tau \bftheta^\perp 
        \mid 
        \tau \in \R\}
    \subset \R^2.
\end{equation*}
Using the bijection $\varphi_\bftheta \colon \R^2 \to \R^2$
defined as
$\varphi_\bftheta(t, \tau) 
\coloneqq 
t \, \bftheta + \tau \, \bftheta^\perp$,
whose inverse is given by 
$\varphi_\bftheta(\bfx)^{-1} 
= (\langle \bfx, \bftheta \rangle, \langle \bfx, \bftheta^\perp \rangle)$,
we parameterize $\ell_{t, \bftheta}$ via
$\tau \mapsto \varphi_\bftheta(t, \tau)$.

For $f \in L^1(\R^2)$, 
we define its \textit{Radon transform} $\Radon [f] \colon \R \times \Sphere \to \R$ as the line integral 
\begin{equation*}
    \Radon [f] (t, \bftheta) 
    \coloneqq 
    \int_{\ell_{t,\bftheta}} 
    f(s) 
    \d s,
    \quad 
    (t,\bftheta) \in \R \times \Sphere,
\end{equation*}
where $\d s$ denotes the arc length element of $\ell_{t,\bftheta}$.
This defines the \textit{Radon operator} $\Radon \colon L^1(\R^2) \to L^1(\R \times \Sphere)$.
For fixed $\bftheta \in \Sphere$, 
we set $\Radon_\bftheta \coloneqq \Radon(\cdot, \bftheta)$, 
which is referred to as the \textit{restricted Radon operator} 
$\Radon_\bftheta \colon L^1(\R^2) \to L^1(\R)$.
The action of the Radon operator is illustrated in Figure~\ref{fig:radon}.
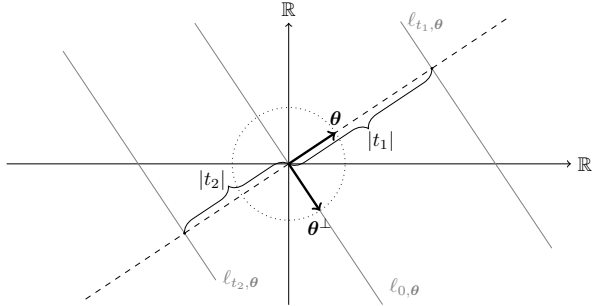
\begin{SCfigure}[2][t]
    \centering
    \scalebox{0.75}{
    \begin{tikzpicture}
        \draw[->] (-5,0) -- (5,0) node[right] {$\R$};
        \draw[->] (0,-2.5) -- (0,2.5) node[above] {$\R$};
        \draw[thin, dashed] (-3*1.2,-2*1.2) -- (3*1.3,2*1.3);  
        \draw[gray] (3,1) -- (3.0507*0.832, 3.0507*0.5547);
        \draw[gray] (3,1) -- (3 + 3*0.5547,1 - 3*0.832);
        \draw[gray] (3.0507*0.832, 3.0507*0.5547) -- (3.0507*0.832 - 1*0.5547, 3.0507*0.5547 + 1*0.832);
        \node at (3.0507*0.832 - 1*0.5547, 3.0507*0.5547 + 1*0.832) [right] {\color{gray}$\ell_{t_1,\bftheta}$};
        \draw[gray] (-4,2) -- (-2.2186*0.832, -2.2186*0.5547);
        \draw[gray] (-2.2186*0.832, -2.2186*0.5547) -- (-2.2186*0.832 + 1*0.5547, -2.2186*0.5547 - 1*0.832);
        \node at (-2.2186*0.832 + 1*0.5547, -2.2186*0.5547 - 1*0.832) [right] {\color{gray}$\ell_{t_2,\bftheta}$};
        \draw[gray] (-0.5547*3,0.832*3) -- (0.5547*3,-0.832*3);
        \node at (0.5547*3,-0.832*3) [above right] {\color{gray}$\ell_{0,\bftheta}$};
        \draw[very thick,->] (0,0) -- (0.5547,-0.832) node[below] {$\bftheta^\perp$};
        \draw[very thick,->] (0,0) -- (0.832,0.5547) node[above] {$\bftheta$};
        \draw[thin, dotted] (0,0) circle (1);
        \node at (3.0507*0.832, 3.0507*0.5547)[circle,fill,inner sep=0.5pt]{};
        \node at (-2.2186*0.832, -2.2186*0.5547)[circle,fill,inner sep=0.5pt]{};
        \draw[decorate,decoration={brace,amplitude=7.5pt,mirror,raise=0pt}] (0,0) -- (3.0507*0.832, 3.0507*0.5547) node[midway,below right,yshift=-0.15cm] {$|t_1|$};
        \draw[decorate,decoration={brace,amplitude=7.5pt,mirror,raise=0pt}] (0,0) -- (-2.2186*0.832, -2.2186*0.5547) node[midway,above left,yshift=1pt,xshift=-2pt] {$|t_2|$};
        \node at (4,0) [anchor=north] {};
        \node at (0,3) [anchor=east] {};
    \end{tikzpicture}
    }
    \caption{Illustration of the bivariate Radon transform
    with distance $t \in \R$
    and normal direction $\bftheta \in \Sphere$.
    The given function is integrated along the lines $\ell_{t, \bftheta}$.
    }
    \label{fig:radon}
\end{SCfigure}
The Radon transform is also well-defined 
for all $f \in L^p(\R^2)$ with $p \geq 1$ 
and $\supp(f) \subseteq \Ball \coloneqq \{ \bfx \in \R^2 \mid \lVert \bfx \rVert \le 1\}$, 
in which case $\Radon [f] \in L^p(\R \times \Sphere)$
with $\supp(\Radon [f]) \subseteq \I \times \Sphere$,
where $\I \coloneqq [-1,1]$.

According to~\cite{Natterer2001}, 
the adjoint 
$\Radon^* \colon L^\infty(\R \times \Sphere) \to L^\infty(\R^2)$ 
of the Radon operator 
$\Radon \colon L^1(\R^2) \to L^1(\R \times \Sphere)$ 
is given by the \emph{back projection}
\begin{equation*}
    \Radon^* [g](\bfx) \coloneqq \int_{\Sphere} g(S_\bftheta(\bfx), \bftheta) \d \sigma_{\Sphere}(\bftheta),
    \quad \bfx \in \R^2,
\end{equation*}
where $\sigma_{\Sphere}$ denotes the surface measure on $\Sphere$.

\subsection{Radon Transform of Measures}

The concept of the Radon transform is now translated to signed, 
regular, finite measures $\mu \in \M(\R^2)$.
For a fixed direction $\bftheta \in \Sphere$, 
we generalize the \textit{restricted Radon transform} $\Radon_\bftheta$ to measures by setting
\begin{equation*}
    \Radon_\bftheta \colon \M(\R^2) \to \M(\R), \quad
    \mu \mapsto (S_\bftheta)_\# \mu = \mu \circ S_\bftheta^{-1},
\end{equation*}
which corresponds to the integration along $\ell_{t,\bftheta}$.
Note that
$
    \Radon_\bftheta[\mu](\R) = \mu(\R^2)
$
for all  $\bftheta \in \Sphere$ and,
thus, 
the mass of $\mu$ is preserved by $\Radon_\bftheta$.
In measure theory, $\Radon_\bftheta$ can be considered as a disintegration family.
Heuristically, 
we may generalize the Radon transform 
by integrating $\Radon_\bftheta$ along $\bftheta \in \Sphere$.
Therefore, we define the \textit{Radon transform} $\Radon \colon \M(\R^2) \to \M(\R \times \Sphere)$ via
\begin{equation}
    \label{eq:radon-meas}
    \Radon [\mu] \coloneqq \Glue_\#[\mu \times u_{\Sphere}]
\end{equation}
with
$
    \Glue(\bfx, \bftheta) \coloneqq (S_\bftheta(\bfx), \bftheta)
$
for $(\bfx, \bftheta) \in \R^2 \times \Sphere$.
Here $\mu \times u_{\Sphere}$ denotes the product measure
between the given $\mu$ 
and the uniform measure 
$u_{\Sphere} \coloneqq \sigma_{\Sphere}/2\pi$
on $\Sphere$.

\begin{proposition}
    Let $\mu \in \M(\R^2)$.
    Then, 
    $\Radon [\mu]$ can be disintegrated 
    into the family $\Radon_\bftheta [\mu]$ 
    with respect to $u_{\Sphere}$,
    i.e., 
    for all continuous $g \in C_0(\R \times \Sphere)$ vanishing at infinity, 
    we have
    \begin{equation*}
        \langle \Radon [\mu], g\rangle 
        = 
        \int_{\Sphere} \langle \Radon_{\bftheta} [\mu], g(\cdot, \bftheta)\rangle \d u_{\Sphere}(\bftheta).
    \end{equation*}
\end{proposition}

\begin{proof}
    By definition in \eqref{eq:radon-meas}, 
    we obtain 
    \begin{align*}
        \langle \Radon [\mu], g\rangle
        & = 
        \int_{\R \times \Sphere} g(t, \bftheta) \d \Glue_\#[\mu \times u_{\Sphere}] (t, \bftheta) 
        = 
        \int_{\Sphere} \int_{\R^2} g(S_\bftheta(\bfx), \bftheta) \d \mu(\bfx) \d u_{\Sphere}(\bftheta) 
        \\
        & = 
        \int_{\Sphere} \int_{\R^2} g(t, \bftheta) \d [(S_\bftheta)_{\#}\mu](t) \d u_{\Sphere}(\bftheta) 
        = 
        \int_{\Sphere} \langle \Radon_{\bftheta} [\mu], g(\cdot, \bftheta)\rangle \d u_{\Sphere}(\bftheta)
    \end{align*}
    using Fubini's theorem. \qed
\end{proof}

One can find the measure-valued Radon transform 
$\Radon \colon \M(\R^2) \to \M(\R \times \Sphere)$ 
as the adjoint of the function-valued adjoint 
$\Radon^* \colon L^\infty (\R \times \Sphere) \to L^\infty(\R^2)$,
similar to the case of distributions with compact support,
cf.~\cite{Ramm1996}.

\begin{proposition} \label{prop:adjRT_measure}
    The Radon transform of $\mu \in \M(\R^2)$ satisfies
    \begin{equation*}
        \langle \Radon [\mu], g\rangle = \langle \mu, \Radon^* [g]\rangle
        \quad \forall \,
        g \in L^\infty(\R \times \Sphere).
    \end{equation*}
\end{proposition}

\begin{proof}
    For all $\mu \in \M(\R^2)$ and $g \in L^\infty(\R \times \Sphere)$, applying Fubini's theorem gives
    \begin{equation*}
        \langle \Radon [\mu], g\rangle = \int_{\R \times \Sphere} g(t, \bftheta)  \d \Glue_\#[\mu \times u_{\Sphere}] (t, \bftheta) = \int_{\R^2} \int_{\Sphere} g(S_\bftheta(\bfx), \bftheta) \d u_{\Sphere}(\bftheta) \d \mu(\bfx).
        \tag*{\qed}
    \end{equation*}
\end{proof}

Note that, for $f \in L^1(\R^2)$ and the Lebesgue measure $\lambda_{\R^2}$ on $\R^2$,
the Radon transform satisfies 
\begin{equation*}
\Radon[f \, \lambda_{\R^2}]  =  \Radon [f] \, \sigma_{\R \times \Sphere},
\end{equation*}
where $\sigma_{\R \times \Sphere}$ denotes the surface measure on $\R \times \Sphere$.
In particular, the Radon transform of an absolutely continuous measure is again absolutely continuous.

\subsection{Radon Transform of Affine Transformations}
\label{sec:Radon_affine_transform}

We now consider the Radon transform of an affinely transformed finite measure $\mu \in \M(\R^2)$.
To this end, let $\bfA \in \GL(2)$ and $\bfy \in \R^2$,
this is, $\bfA$ is contained in the general linear group $\GL$ of regular matrices.
We define $\mu_{\bfA,\bfy} \in \M(\R^2)$ via
\begin{equation}
    \label{eq:aff-meas}
    \mu_{\bfA,\bfy} 
    \coloneqq 
    (\bfA \cdot + \bfy)_{\#} \mu = \mu \circ (\bfA^{-1}(\cdot-\bfy)).
\end{equation}

\begin{proposition} \label{prop:RT_transformation}
    For any $\bftheta \in \Sphere$, the restricted Radon transform satisfies
    \begin{equation*}
        \Radon_{\bftheta} [\mu_{\bfA,\bfy}] 
        = 
        (\lVert \bfA^\top \bftheta\rVert \cdot + \langle \bfy, \bftheta\rangle)_{\#} 
        \Radon_{\frac{\bfA^\top \bftheta}{\lVert \bfA^\top \bftheta \rVert}}[\mu]
        = 
        \Radon_{\frac{\bfA^\top \bftheta}{\lVert \bfA^\top \bftheta \rVert}}[\mu] 
        \circ
        \Bigl(\tfrac{\cdot - \langle \bfy, \bftheta\rangle}{\lVert \bfA^\top \bftheta \rVert}\Bigr).
    \end{equation*}
\end{proposition}

\begin{proof}
    Direct calculations yield
    \begin{align*}
        \Radon_{\bftheta} [\mu_{\bfA,\bfy}]
        &= 
        (S_\bftheta)_{\#}[(\bfA \cdot + \bfy)_{\#} \mu] 
        = 
        (\langle \bfA \cdot + \bfy, \bftheta\rangle)_{\#} \mu 
        =
        (\langle \cdot, \bfA^\top \bftheta\rangle + \langle \bfy, \bftheta\rangle)_{\#} \mu 
        \\
        &=
        \bigl(
            \lVert \bfA^\top \bftheta\rVert 
            \bigl\langle \cdot, \tfrac{\bfA^\top \bftheta}{\lVert \bfA^\top \bftheta\rVert}\bigr\rangle 
            +
            \langle \bfy, \bftheta\rangle
        \bigr)_{\#} \mu
        =
        (\lVert \bfA^\top \bftheta\rVert \cdot + \langle \bfy, \bftheta\rangle)_{\#} 
        \Radon_{\frac{\bfA^\top \bftheta}{\lVert \bfA^\top \bftheta \rVert}}[\mu],
    \end{align*}
    and the proof is complete.
    \qed
\end{proof}

\begin{table}[t]
    \caption{Summary of common transformations for $\mu \in \M(\R^2)$ 
        with $a,b > 0$ and $c, \varphi \in \R$.
        The unit circle is parameterized by
        $\bftheta(\vartheta) \coloneqq (\cos(\vartheta), \sin(\vartheta))^\top$.
        The Radon transform for the left half of $\Sphere$ follows by symmetry.}
    \label{tab:radon-aff-trans}
    \begin{minipage}{\linewidth}
    \resizebox{\linewidth}{!}{
        \begin{tabular}{l@{\enspace}l@{\enspace}l@{\enspace}l}
            \toprule
            transformation 
            & 
            $\bfA$
            &
            $\bfy$
            &
            $\Radon_{\bftheta(\vartheta)} [\mu_{\bfA, \bfy}]$,
            $\vartheta \in (-\tfrac{\pi}{2}, \tfrac{\pi}{2})$
            \\
            \midrule
            translation 
            & 
            $\bfI$ 
            &
            $\R^2$ 
            &
            $\Radon_{\bftheta(\vartheta)} [\mu] \circ (\cdot - \langle \bfy, \bftheta(\vartheta) \rangle)$ 
            \\[1ex]
            rotation
            & 
            $\bigl(\begin{smallmatrix} \cos(\varphi) & -\sin(\varphi) \\ \sin(\varphi) & \hphantom{-}\cos(\varphi) \end{smallmatrix}\bigr)$
            &
            $\bfzero$ 
            & 
            $\Radon_{\bftheta(\vartheta - \varphi)} [\mu]$ 
            \\[1ex]
            reflection 
            &
            $\bigl(\begin{smallmatrix} \cos(\varphi) & \hphantom{-}\sin(\varphi) \\ \sin(\varphi) & -\cos(\varphi) \end{smallmatrix}\bigr)$
            &
            $\bfzero$  
            &
            $\Radon_{\bftheta(\varphi - \vartheta)} [\mu]$ 
            \\[1ex]
            anisotropic scaling 
            & 
            $\bigl(\begin{smallmatrix} a & 0 \\ 0 & b \end{smallmatrix}\bigr)$
            &
            $\bfzero$  
            &
            $\Radon_{\bftheta(\arctan(\frac{b}{a} \tan(\vartheta)))}[\mu] \circ ([a^2 \cos^2(\vartheta) + b^2 \sin^2(\vartheta)]^{-1/2} \cdot)$ 
            \\[1ex]
            vertical shear
            & 
            $\bigl(\begin{smallmatrix} 1 & 0 \\ c & 1 \end{smallmatrix}\bigr)$
            &
            $\bfzero$ 
            & 
            $\Radon_{\bftheta(\arctan(c + \tan(\vartheta)))}[\mu] \circ ([1 + c^2 \cos^2(\vartheta) + c \sin(2\vartheta)]^{-1/2} \cdot)$
            \\
            \bottomrule
        \end{tabular}
    }
    \end{minipage}
\end{table}

The effect of common affine transformations on the Radon transform
is given in Table~\ref{tab:radon-aff-trans}.
In order to describe the deformation with respect to $\bftheta$,
we over-parameterize the unit circle $\Sphere$ via
$\bftheta(\vartheta) \coloneqq (\cos(\vartheta), \sin(\vartheta))^\top$, $\vartheta \in \R$.
As by Proposition~\ref{prop:RT_transformation},
an affine transformation essentially causes a transition and dilation
of the transformed measure 
together with a non-affine remapping in $\bftheta$.

\section{Optimal Transport-Based Transforms}
\label{sec:ot-dist}

The aim of the following is to introduce an image distance
that is unaware of affine transformations.
Methodologically,
we rely on the \emph{Radon cumulative distribution transform} (R-CDT)
introduced in \cite{Kolouri2016},
which allows to utilize the fast-to-compute, one-dimensional Wasserstein distance
in the context of image processing 
due to a Radon-based slicing technique.
As the R-CDT is not invariant under affine transformation by itself,
we propose a two-step normalization scheme,
which is essentially grounded on our observations 
regarding the Radon transform under affine transformations
in §~\ref{sec:Radon_affine_transform}.
Finally,
we study the linear separability
of affinely transformed image classes
by our novel normalized R-CDT.

\subsection{R-CDT for Measures}

The R-CDT traces back to Kolouri et al.\ \cite{Kolouri2016}
and transforms smooth, bivariate density functions.
In difference to \cite{Kolouri2016},
we introduce the concept
for arbitrary probability measures,
similar to~\cite{DiazMartin2024}.
In a first step,
we consider probability measures $\P(\R) \subset \M(\R)$ 
defined on the real line.
For $\mu \in \P(\R)$,
the \emph{cumulative distribution function}
$F_\mu \colon \R \to [0,1]$
is given by $F_\mu(t) \coloneqq \mu((-\infty,t])$, $t \in \R$.
Its generalized inverse,
known as \emph{quantile function},
reads as
\begin{equation*}
    F_\mu^{[-1]}(t) \coloneqq \inf \{s \in \R \mid F_\mu(s) > t\},
    \quad t \in \R.
\end{equation*}
Based on a reference measure $\rho \in \P(\R)$
that does not give mass to atoms,
e.g.,
the uniform distribution
$u_{[0,1]}$ on $[0,1]$,
we define
the {\em cumulative distribution transform}
$\widehat{\mu} \colon \R \to \R$, in short CDT, via
\begin{equation*}
    \widehat{\mu} \coloneqq F_\mu^{[-1]} \circ F_\rho.
\end{equation*}

For any convex cost function $c \colon \R \to [0,\infty)$,
the CDT
(with respect to $\rho$) 
solves the Monge--Kantorovich transportation problem
\cite{Villani2003},
this is,
\begin{equation*}
    \widehat{\mu} = \argmin_{T_\# \rho = \mu} \int_\R c(s-T(s))  \d \rho(s),
\end{equation*}
where the minimum is taken over all measurable functions $T \colon \R \to \R$.
In other words,
$\widehat{\mu} \colon \R \to \R$ is an optimal Monge map
transporting $\rho$ to $\mu$
while minimizing the cost.
If $\mu \in \P_2(\R)$, i.e., $\mu$ has finite 2nd moment,
then $\widehat{\mu}$ is square integrable
with respect to $\rho$,
i.e.,
$\widehat{\mu} \in L^2_\rho(\R)$.
Moreover,
for $\mu, \nu \in \P_2(\R)$,
the norm distance
\begin{equation*}
    \lVert \widehat{\mu} - \widehat{\nu} \rVert_\rho
    \coloneqq
    \Bigl(
        \int_\R
        \lvert \mu(t) - \nu(t) \rvert^2
        \d \rho(t)
    \Bigr)^{\frac{1}{2}}
\end{equation*}
equals the well-established Wasserstein-2 distance 
\cite{Villani2003}.

To deal with a probability measure $\mu \in \P(\R^2)$
defined on the plane,
we first determine the Radon transform $\Radon[\mu] \in \M(\R \times \Sphere)$ 
with its disintegration family $\{\Radon_\bftheta[\mu] \in \P(\R) \mid \bftheta \in \Sphere \}$.
Then,
for each fixed $\bftheta \in \Sphere$, 
we consider the CDT $\widehat{\Radon}_\bftheta [\mu]$
(with respect to the same reference measure $\rho \in \P(\R)$
for all $\bftheta \in \Sphere$)
of the Radon projection $\Radon_\bftheta [\mu]$, 
yielding the \emph{R-CDT}\, 
$\widehat{\Radon}[\mu] \colon \R \times \Sphere \to \R$ of $\mu$ via
\begin{equation*}
    \widehat{\Radon} [\mu](t,\bftheta) 
    \coloneqq \widehat{\Radon}_\bftheta [\mu](t),
    \quad  (t,\bftheta) \in \R \times \Sphere.
\end{equation*}
If $\mu \in \P_2(\R^2)$,
then the Radon projection $\R_\bftheta [\mu] \in \P_2(\R)$
has finite 2nd moment as well.
Consequently,
$\widehat{\Radon}[\mu] \in L_{\rho \times u_{\Sphere}}^2(\R \times \Sphere)$.
For $\mu, \nu \in \P_2(\R^2)$,
the norm distance
\begin{equation*}
    \lVert 
        \widehat{\Radon}[\mu] - \widehat{\Radon}[\nu]
    \rVert_{\rho \times u_{\Sphere}}
    \coloneqq
    \Bigl(
        \int_{\Sphere} \int_\R
        \lvert 
            \widehat{\Radon}[\mu](t, \bftheta)
            -
            \widehat{\Radon}[\nu](t, \bftheta)
        \rvert^2
        \d \rho(t) \d u_{\Sphere}(\bftheta)
    \Bigr)^{\frac{1}{2}}
\end{equation*}
resembles the so-called sliced Wasserstein-2 distance \cite{Bonneel2015}.

\subsection{Normalized R-CDT}

The R-CDT is by itself not invariant under affine transformations,
which emerge in various applications.
More precisely,
the R-CDT inherits the behavior of the Radon transform
observed in §~\ref{sec:Radon_affine_transform}.
Notice that
the translation and dilation of $\Radon_\bftheta[\mu]$
causes a horizontal shift (addition of a constant)
and a scaling (multiplication with a constant)
of $\widehat{\Radon}_\bftheta [\mu]$,
respectively.
In the first normalization step,
we revert this effects by
ensuring zero mean and unit standard deviation
of the R-CDT projection.
More precisely,
we define the {\em normalized R-CDT} (NR-CDT)
$\NRCDT [\mu] \colon \R \times \Sphere \to \R$ 
of $\mu \in \P_2(\R^2)$ via
\begin{equation*}
    \NRCDT [\mu](t,\bftheta) 
    \coloneqq
    \NRCDT_\bftheta[\mu](t)
    \coloneqq
    \frac
    {\widehat{\Radon}_\bftheta[\mu](t) 
    - 
    \mean(\widehat{\Radon}_\bftheta[\mu])}
    {\std(\widehat{\Radon}_\bftheta[\mu])},
    \quad (t,\bftheta) \in \R \times \Sphere,
\end{equation*}
where, for $g \in \Lebesgue^2_\rho(\R)$,
\begin{equation*}
    \mean(g) 
    \coloneqq 
    \int_\R g(s) \d \rho(s),
    \qquad
    \std(g) 
    \coloneqq 
    \Bigr(
        \int_\R 
        \lvert g(s) - \mean(g) \rvert^2 
        \d \rho(s)
    \Bigr)^{\frac{1}{2}}.
\end{equation*}

To ensure
that the NR-CDT is well defined,
we have to guarantee
that the standard deviation of the R-CDT projection does not vanish.
For this,
we restrict ourselves to measures
whose supports are not contained in a straight line.
More precisely,
we consider the class
\begin{equation*}
    \P_c^*(\R^2) 
    \coloneqq 
    \{\mu \in \P(\R^2) \mid \supp(\mu) \subsubset \R^2 \land \dim(\supp(\mu)) > 1\}
    \subset
    \P_2(\R^2).
\end{equation*}
Here,
$\subsubset$ denotes a compact subset, 
and $\dim$ the dimension 
of the affine hull.
For these,
the standard deviation of the restricted Radon transform
is bounded away from zero and
cannot vanish.

\begin{proposition} 
    \label{prop:sigma_bounded}
    Let $\mu \in \P_c^*(\R^2)$.
    Then, there exists a constant $c > 0$ such that
    \begin{equation*}
        \std(\widehat{\Radon}_\bftheta [\mu]) \geq c
        \quad \forall \, \bftheta \in \Sphere.
    \end{equation*}
\end{proposition}

For the proof,
we first show the following continuity.

\begin{lemma}
    For fixed $\mu \in \P_c^*(\R^2)$,
    the functions
    $\bftheta \in \Sphere \mapsto \mean(\widehat{\Radon}_\bftheta [\mu]) \in \R$
    and
    $\bftheta \in \Sphere \mapsto \std(\widehat{\Radon}_\bftheta [\mu]) \in \R_{\geq 0}$
    are continuous.
\end{lemma}

\begin{proof}
    We rewrite the mean as
    \begin{equation*}
        \mean(\widehat{\Radon}_\bftheta[\mu])
        =
        \int_\R \widehat{\Radon}_\bftheta [\mu] (t) \d \rho(t)
        =
        \int_\R t  \d \Radon_\bftheta [\mu] (t)
        =
        \int_{\R^2} \langle \bfx, \bftheta \rangle \d \mu(\bfx).
    \end{equation*}
    Since the integrand is continuous in $\bftheta$
    and uniformly bounded by 
    $\lvert \langle \cdot, \bftheta \rangle \rvert \le \lVert \cdot \rVert$,
    the dominated convergence yields the assertion.
    Analogously,
    we have
    \begin{align*}
        \std(\widehat{\Radon}_\bftheta [\mu])
        & =
        \Bigl(
        \int_\R 
        \lvert
            \widehat{\Radon}_\bftheta [\mu] (t)
            -
            \mean(\widehat{\Radon}_\bftheta [\mu])
        \rvert^2
        \d \rho(t)
        \Bigr)^{\frac{1}{2}}
        \\
        & =
        \Bigl(
        \int_{\R^2} 
        \lvert
            \langle \bfx, \bftheta \rangle
            -
            \mean(\widehat{\Radon}_\bftheta [\mu])
        \rvert^2
        \d \mu(\bfx)
        \Bigr)^{\frac{1}{2}}.
    \end{align*}
    The integrand is again continuous in $\bftheta$
    and uniformly bounded by
    \begin{equation*}
        \lvert
            \langle \cdot, \bftheta \rangle
            -
            \mean(\widehat{\Radon}_\bftheta [\mu])
        \rvert^2
        \le 
        2 \lVert \cdot \rVert^2 
        +
        2 \max_{\bftheta \in \Sphere} \;
        (\mean(\widehat{\Radon}_\bftheta [\mu]))^2;
    \end{equation*}
    thus,
    the standard deviation is continuous by dominated convergence.
    \qed
\end{proof}

\begin{proof}[Proposition~\ref{prop:sigma_bounded}]
    Assume the contrary,
    this is,
    $c = 0$.
    Then,
    due to the continuity of $\bftheta \mapsto \std(\widehat{\Radon}_\bftheta [\mu])$,
    there exists a minimizing and convergent sequence in $\Sphere$
    whose limit $\bftheta$ is attained and satisfies
    $\std(\widehat{\Radon}_\bftheta [\mu]) = 0$,
    i.e.,
    \begin{equation*}
        \int_{\R^2} \lvert \langle \bfx, \bftheta \rangle - \mean(\widehat{\Radon}_\bftheta [\mu]) \rvert^2 \d \mu(\bfx) = 0.
    \end{equation*}
    Hence, 
    the support of $\mu$ is contained in the line 
    $\{\bfx \in \R^2 \mid \langle \bfx, \bftheta \rangle = \mean(\widehat{\Radon}_\bftheta [\mu])\}$
    in contradiction to $\mu \in \P_c^*(\R^2)$.
    \qed
\end{proof}

The NR-CDT is nearly invariant under affine transformations
up to bijective remappings of the directions,
i.e.,
up to a resorting of the family 
$\{\NRCDT_\bftheta[\mu] \mid \bftheta \in \Sphere\}$.

\begin{proposition}
    Let $\mu \in \P_c^*(\R^2)$,
    $\bfA \in \GL(2)$,
    $\bfy \in \R^2$,
    and $\mu_{\bfA, \bfy}$ as in \eqref{eq:aff-meas}.
    Then, for any $\bftheta \in \Sphere$,
    the NR-CDT satisfies
    \begin{equation*}
        \NRCDT_\bftheta[\mu_{\bfA, \bfy}] 
        =
        \NRCDT_{\frac{\bfA^\top \bftheta}{\lvert\bfA^\top \bftheta\rvert}} [\mu].
    \end{equation*}
\end{proposition}

\begin{proof}
    Transferring Proposition~\ref{prop:RT_transformation} to the CDT space,
    we have
    \begin{equation*}
        \widehat{\Radon}_\bftheta [\mu_{\bfA, \bfy}](t)
        =
        \lVert \bfA^\top \bftheta \rVert \,
        \widehat{\Radon}_{h_\bfA(\bftheta)} [\mu](t)
        + 
        \langle \bfy, \bftheta \rangle
    \end{equation*}
    with the bijection $h_\bfA (\bftheta) \coloneqq (\bfA^\top \bftheta) / \lVert \bfA^\top \bftheta \rVert$,
    $\bftheta \in \Sphere$;
    so that
    \begin{equation*}
        \mean(\widehat{\Radon}_\bftheta [\mu_{\bfA,\bfy}]) 
        =
        \lVert \bfA^\top \bftheta \rVert  
        \mean (\widehat{\Radon}_{h_\bfA(\bftheta)} [\mu] ) 
        +
        \langle \bfy, \bftheta \rangle
    \end{equation*}
    and
    \begin{equation*}
        \std(\widehat{\Radon}_\bftheta [\mu_{\bfA, \bfy}]) 
        =
        \lVert \bfA^\top \bftheta \rVert 
        \std(\widehat{\Radon}_{h_\bfA(\bftheta)} [\mu]).
    \end{equation*}
    Consequently,
    \begin{equation*}
        \NRCDT_\bftheta [\mu_{\bfA, \bfy}](t) 
        =
        \frac
        {
            \widehat{\Radon}_{h_\bfA(\bftheta)} [\mu](t)
            - 
            \mean(\widehat{\Radon}_{h_\bfA(\bftheta)} [\mu])
        }{
            \std(\widehat{\Radon}_{h_\bfA(\bftheta)} [\mu])
        }
        =
        \NRCDT_{h_\bfA(\bftheta)} [\mu](t).
        \tag*{\qed}
    \end{equation*}
\end{proof}

\subsection{Max-Normalized R-CDT}

In the final normalization step,
we treat the resorting of 
$\{\NRCDT_\bftheta[\mu] \mid \bftheta \in \Sphere\}$.
Since the underlying mapping is unknown in general
and cannot be reverted,
we propose to take the supremum over all directions.
More precisely,
for $\mu \in \P_c^*(\R^2)$, 
we define its \emph{max-normalized R-CDT} (\mNRCDT) $\maxNRCDT [\mu] \colon \R \to \R$ 
via
\begin{equation*}
    \maxNRCDT [\mu](t) 
    \coloneqq  
    \sup_{\bftheta \in \Sphere} \NRCDT_\bftheta [\mu](t),
    \quad  t \in \R.
\end{equation*}
We show that
$\maxNRCDT$ maps a given measure
to a bounded function
so that
the \mNRCDT{} space $\maxNRCDT[\P_c^*(\R^2)]$
is contained in $L_\rho^\infty(\R)$ for the underlying reference measure $\rho \in \P(\R)$.

\begin{proposition}
    Let $\mu \in \P_c^*(\R^2)$.
    Then, $\maxNRCDT [\mu] \in L^\infty_\rho(\R)$.
\end{proposition}

\begin{proof}
    The restricted Radon operator cannot enlarge the size of the support
    $\diam(\mu) \coloneqq \sup_{\bfx, \bfy \in \supp(\mu)} \, \lVert \bfx - \bfy \rVert$,
    this is,
    $\diam(\Radon_\bftheta [\mu]) \le \diam(\mu)$.
    Moreover, the range of $\widehat{\Radon}_\bftheta [\mu]$ coincides
    with the support of $\Radon_\bftheta [\mu]$.
    Using that the mean lies in the convex hull of the support,
    we thus have
    \begin{equation*}
        \lvert 
        \widehat{\Radon}_\bftheta [\mu] (t) 
        - 
        \mean(\widehat{\Radon}_\bftheta [\mu])
        \rvert
        \le 
        \diam(\mu)
        \quad
        \forall \, \bftheta \in \Sphere.
    \end{equation*}
    Since $\mu \in \P_c^*(\R^2)$,
    Proposition~\ref{prop:sigma_bounded} gives
    $c \coloneqq \min_{\bftheta \in \Sphere} \std(\widehat{\Radon}_\bftheta [\mu]) > 0$.
    Thus,
    the \mNRCDT{} is bounded by
    $\lvert \maxNRCDT [\mu](t) \rvert \le \diam(\mu) / c$
    for all $t \in \R$.
    \qed
\end{proof}

With the \mNRCDT{},
we accomplish our objective
to define a transport-based transform
that is invariant under affine transformations.

\begin{proposition}
    \label{prop:inv-mnr-dct}
    Let $\mu \in \P_c^*(\R^2)$,
    $\bfA \in \GL(2)$,
    $\bfy \in \R^2$,
    and $\mu_{\bfA, \bfy}$ as in~\eqref{eq:aff-meas}.
    Then,
    the \mNRCDT{} satisfies
    $\maxNRCDT[\mu_{\bfA, \bfy}] = \maxNRCDT[\mu]$.
\end{proposition}

\begin{proof}
    Since the mapping 
    $h_\bfA(\bftheta) \coloneqq (\bfA^\top \bftheta) / \lVert \bfA^\top \bftheta \rVert$
    is a bijection on $\Sphere$,
    we obtain
    \begin{equation*}
        \maxNRCDT[\mu_{\bfA,\bfy}](t)
        =
        \sup_{\bftheta \in \Sphere} \NRCDT_\bftheta [\mu_{\bfA, \bfy}](t) 
        =
        \sup_{\bftheta \in \Sphere} \NRCDT_{h_\bfA(\bftheta)} [\mu](t) 
        =
        \maxNRCDT [\mu](t).
        \tag*{\qed}
    \end{equation*}        
\end{proof}

The invariance 
under affine transformations 
immediately yields the linear separability
of affine measure classes,
which originate from a single template.

\begin{theorem}
    \label{thm:sep-max-nrcdt}
    For template measures $\mu_0, \nu_0 \in \P_c^*(\R^2)$ with
    \begin{equation*}
        \maxNRCDT [\mu_0] \neq \maxNRCDT [\nu_0]
    \end{equation*}
    consider the classes
    \begin{subequations}
    \label{eq:aff-class}
    \begin{align}
        \F 
        &=
        \bigl\{\mu \in \P(\R^2) \mid \exists \, \bfA \in \GL(2), \, \bfy \in \R^2 \colon \mu = (\bfA \cdot + \bfy)_\# \mu_0\bigr\},
        \\
        \G 
        &= 
        \bigl\{\nu \in \P(\R^2) \mid \exists \, \bfA \in \GL(2), \, \bfy \in \R^2 \colon \nu = (\bfA \cdot + \bfy)_\# \nu_0\bigr\}.
    \end{align}
    \end{subequations}
    Then, $\F$ and $\G$ are linearly separable in \mNRCDT{}~space.
\end{theorem}

\begin{proof}
    Due to the affine construction of $\F$ and $\G$,
    Proposition~\ref{prop:inv-mnr-dct} yields
    $\maxNRCDT [\F] = \bigl\{\maxNRCDT [\mu_0]\bigr\}$
    and $\maxNRCDT [\G] = \bigl\{\maxNRCDT [\nu_0]\bigr\}$.
    Hence, 
    the assumption
    $\maxNRCDT [\mu_0] \neq \maxNRCDT [\nu_0]$
    implies the linear separability of 
    $\maxNRCDT [\F]$ and $\maxNRCDT [\G]$
    in $L^\infty_\rho(\R)$.
    \qed
\end{proof}

\section{Numerical experiments}
\label{sec:num-ex}

By the following proof-of-concept experiments,
we support our linear separability result in Theorem~\ref{thm:sep-max-nrcdt}
with numerical evidence. 
For this,
the proposed \mNRCDT{} is implemented in Julia%
\footnote{The Julia Programming Language -- Version 1.9.2 
(\url{https://docs.julialang.org}).}.
All experiments%
\footnote{The code will be available at GitHub:
\url{https://github.com/DrBeckmann/NR-CDT}.}
are performed on an off-the-shelf MacBookPro 2020 
with Intel Core i5 Chip (4‑Core CPU, 1.4~GHz) and 8~GB~RAM.

\paragraph{Datasets.}
For our simulations,
we rely on two datasets.
For academic purposes,
the first dataset
is based on (up to) three synthetic template symbols,
which are randomly translated, rotated, dilated, and sheared,
cf.\ Figure~\ref{fig:templates_syn}.
In this manner, we construct perfect affine classes
as needed for our theory,
see~\eqref{eq:aff-class}.
For a more realistic scenario,
we also consider the LinMNIST dataset \cite{Beckmann2024}
consisting of affinely transformed MNIST digits~\cite{Deng2012},
cf.\ Figure~\ref{fig:templates_mnist}.
In contrast to the first dataset,
this data does not originate from a common ground truth.
Therefore,
the second dataset can be considered as a
collection of imperfect affine classes.

\begin{figure}[t]
\begin{minipage}[t]{0.49\linewidth}
    \begin{tabular}{c @{\hspace{-0.15cm}} c @{\hspace{-0.15cm}} c}
        class 1 & class 2 & class 3 \\
        \includegraphics[width=0.33\linewidth, clip=true, trim=460pt 285pt 30pt 30pt]{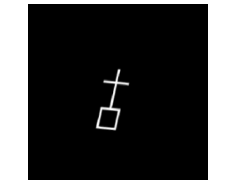} 
        & \includegraphics[width=0.33\linewidth, clip=true, trim=260pt 285pt 230pt 30pt]{Images/synth_data.pdf}
        & \includegraphics[width=0.33\linewidth, clip=true, trim=260pt 285pt 230pt 30pt]{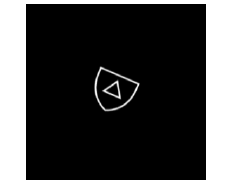}\\
        \includegraphics[width=0.33\linewidth, clip=true, trim=60pt 160pt 430pt 150pt]{Images/synth_data.pdf} 
        & \includegraphics[width=0.33\linewidth, clip=true, trim=460pt 160pt 30pt 150pt]{Images/synth_data.pdf}
        & \includegraphics[width=0.33\linewidth, clip=true, trim=60pt 160pt 430pt 150pt]{Images/synth_data_shield.pdf}
    \end{tabular}
    \captionof{figure}{Samples of the academic dataset 
    consisting of randomly affine-transformed
    synthetic template images.}
    \label{fig:templates_syn}
\end{minipage}
\hfill
\begin{minipage}[t]{0.49\linewidth}
    \begin{tabular}{c @{\hspace{-0.15cm}} c @{\hspace{-0.15cm}} c}
        class 1 & class 5 & class 7 \\
        \includegraphics[width=0.325\linewidth, clip=true, trim=60pt 285pt 430pt 30pt]{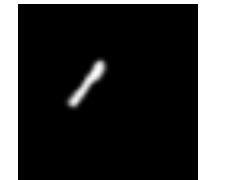}
        & \includegraphics[width=0.325\linewidth, clip=true, trim=460pt 30pt 30pt 282pt]{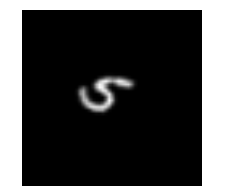}
        & \includegraphics[width=0.325\linewidth, clip=true, trim=260pt 285pt 230pt 30pt]{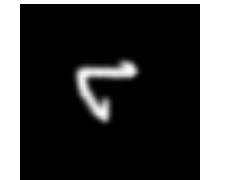} \\
        \includegraphics[width=0.325\linewidth, clip=true, trim=60pt 30pt 430pt 282pt]{Images/plt_mnist_rand_data_noise-1.pdf}
        & \includegraphics[width=0.325\linewidth, clip=true, trim=260pt 30pt 230pt 277pt]{Images/plt_mnist_rand_data_noise-3.pdf}
        & \includegraphics[width=0.325\linewidth, clip=true, trim=460pt 157pt 30pt 150pt]{Images/plt_mnist_rand_data_noise-2.pdf}
    \end{tabular}
    \captionof{figure}{Samples of the LinMNIST dataset 
    (random choices of ones, fifths and sevens) 
    based on affine-transformed MNIST digits.}
    \label{fig:templates_mnist}
\end{minipage}
\end{figure}

\begin{figure}[t]
    \begin{minipage}{0.49\linewidth}
    \captionof{table}{Accuracy of nearest neighbor classification
    for the academic dataset
    with 10 images per class
    and the LinMNIST dataset
    with 50 images per class.}
    \label{tab:nearest_neighbor}
    \resizebox{\linewidth}{!}{\begin{tabular}{l @{\quad} l @{\quad} l @{\quad} l @{\quad} l}
        \toprule
        num. & \multicolumn{2}{l}{academic}  & \multicolumn{2}{l}{LinMNIST} \\
        angles & $\|\cdot\|_\infty$ & $\|\cdot\|_2$ & $\|\cdot\|_\infty$ & $\|\cdot\|_2$\\
        \midrule
        2 & 0.76 & 1.00 &  $0.540\pm0.126$ & $0.591\pm0.130$ \\
        4 & 0.83 & 0.93 &  $0.565\pm0.104$ & $0.642\pm0.105$\\
        8 & 1.00 & 1.00 & $0.644\pm0.120$ & $0.726\pm0.119$\\
        16 & 1.00 & 1.00 & $0.654\pm0.115$ & $0.726\pm0.120$\\
        32 & 1.00 & 1.00 & $0.655\pm0.121$ & $0.721\pm0.120$\\
        64 & 1.00 & 1.00 & $0.656\pm0.119$ & $0.724\pm0.119$\\
        128 & 1.00 & 1.00 & $0.656\pm0.121$ & $0.721\pm0.116$\\
        \bottomrule
    \end{tabular}}
    \end{minipage}
    \hfill
    \begin{minipage}{0.49\linewidth}
        \includegraphics[width=\linewidth, clip=true, trim=20pt 20pt 500pt 1075pt]{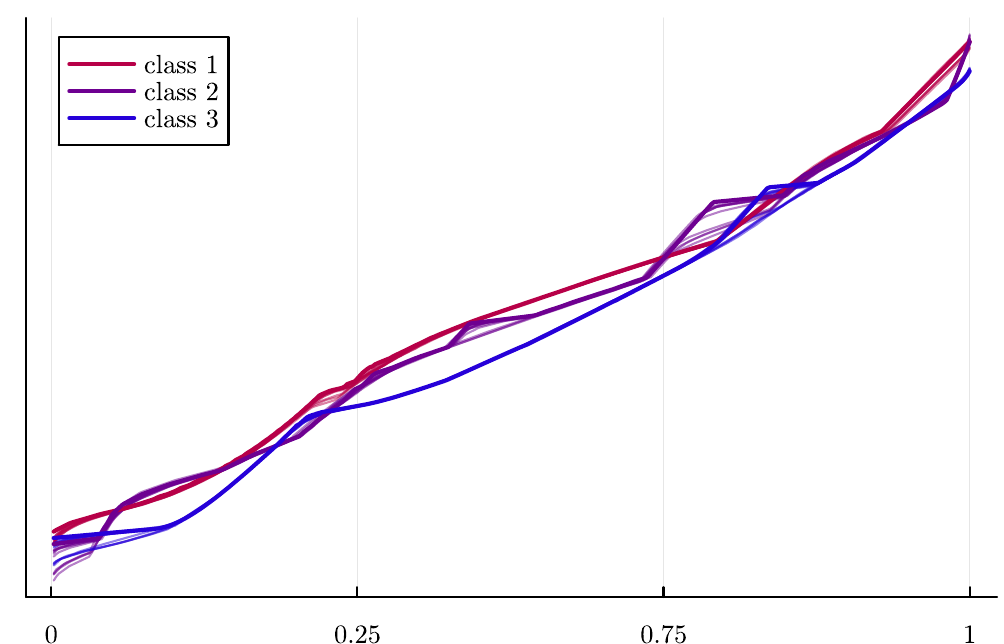}
        \captionof{figure}{Visualization of \mNRCDT{} 
        for the academic dataset
        and 128 angles in $[0,\pi)$.}
        \label{fig:nearest_neighbor}
    \end{minipage}
\end{figure}

\subsection{Nearest Neighbour Classification}

In the first experiment,
we aim to validate the theoretical result from Theorem~\ref{thm:sep-max-nrcdt}.
Looking at the proof,
we recall that
\mNRCDT{} maps each entire affine class
to a single point.
The easiest way for classification is the nearest neighbor method,
which can be immediately generalized to an arbitrary number of classes.
For the first dataset,
we use the template symbols as references 
and classify all class members based on
the nearest neighbour rule with respect to
the Chebychev and Euclidean norm,
cf.~Table~\ref{tab:nearest_neighbor} (columns 2 and 3)
for qualitative results.
For illustration, the \mNRCDT{} of all considered classes
are depicted in Figure~\ref{fig:nearest_neighbor}.
In theory, the classes should yield three curves.
However,
due to approximation errors,
we observe slight perturbations.
For the second dataset,
since we have no templates,
we iteratively select one instance per class as reference 
and classify the remaining class members again based on
the nearest neighbour rule.
Thereon, we compute the mean and standard deviation
of the achieved accuracy,
see~Table~\ref{tab:nearest_neighbor} (columns 4 and 5).
For the discretization of the \mNRCDT{},
we use $2$ to $128$ angles in $[0,\pi)$,
reported in column 1 of~Table~\ref{tab:nearest_neighbor}.
As expected,
due to Theorem~\ref{thm:sep-max-nrcdt},
the classification of the first dataset
is (nearly) perfect;
remarkable, already for a very small number of chosen angles.
For the LinMNIST dataset,
the achieved accuracy ranges from $55\%$ to $73\%$,
which is still significantly better than random guessing,
achieving an accuracy of $33\%$
as we deal with a three class problem.
Let us stress that perfect classification is not to be expected
since LinMNIST does not satisfy our theoretical assumptions.

\subsection{Support Vector Machine Classification}

\begin{table}[t]
    \caption{
    Classification accuracy (mean$\pm$std based on $10$-fold cross validation) 
    for the academic dataset.
    The first two classes in Fig.~\ref{fig:templates_syn} are used
    with different class sizes
    and numbers of equispaced angles in $[0,\pi)$.
    }
    \label{tab:max_comparison_syn}
    \resizebox{\linewidth}{!}{
    \begin{tabular}{l @{\quad} l @{\quad} l @{\quad} l @{\quad} l @{\quad} l @{\quad} l @{\quad} l @{\quad} l @{\quad} l @{\quad}}
    \toprule 
    class & Euclidean & \multicolumn{4}{l}{R-CDT} & \multicolumn{4}{l}{\mNRCDT{}} \\
    size & & 2 & 4 & 8 & 16 & 2 & 4 & 8 & 16 \\
    \midrule
    10 & $0.489\pm0.097$ & $0.644\pm0.112$ & $0.572\pm0.079$ & $0.567\pm0.086$ & $0.561\pm0.076$ & $0.872\pm0.114$ & $0.806\pm0.092$ & $0.944\pm0.082$ & $\mathbf{0.989\pm0.023}$\\
    30 & $0.520\pm0.092$ & $0.825\pm0.128$ & $0.728\pm0.089$ & $0.704\pm0.065$ & $0.704\pm0.080$ & $0.956\pm0.094$ & $0.979\pm0.051$ & $\mathbf{1.000\pm0.000}$ & $\mathbf{1.000\pm0.000}$  \\
    90 & $0.551\pm0.028$ & $0.962\pm0.036$ & $0.952\pm0.056$ & $0.974\pm0.043$ & $0.982\pm0.041$ & $0.996\pm0.008$ & $0.990\pm0.009$ & $\mathbf{1.000\pm0.000}$ & $\mathbf{1.000\pm0.000}$\\
    270 & $0.610\pm0.021$ & $0.997\pm0.006$ & $0.999\pm0.001$ & $\mathbf{1.000\pm0.000}$ & $\mathbf{1.000\pm0.000}$ & $\mathbf{1.000\pm0.000}$ & $\mathbf{1.000\pm0.000}$ & $\mathbf{1.000\pm0.000}$ & $\mathbf{1.000\pm0.000}$\\
    \bottomrule
    \end{tabular}}
    \vspace*{-15pt}
\end{table}

\begin{table}[t]
\caption{
Classification accuracy (mean$\pm$std based on $10$-fold cross validation)
for the LinMNIST dataset.
The first and last class in Fig.~\ref{fig:templates_mnist} are used
with different class sizes
and numbers of equispaced angles in $[0,\pi)$.
}
\label{tab:max_comparison_mnist}
\resizebox{\linewidth}{!}{%
\begin{tabular}{l @{\quad} l @{\quad} l @{\quad} l @{\quad} l @{\quad} l @{\quad} l @{\quad} l @{\quad} l @{\quad} l @{\quad}}
\toprule
class & Euclidean & \multicolumn{4}{l}{R-CDT} & \multicolumn{4}{l}{\mNRCDT{}} \\
size & & 4 & 8 & 16 & 32 & 4 & 8 & 16 & 32 \\
\midrule 
10 & $0.478\pm0.070$ & $0.611\pm0.047$ & $0.556\pm0.027$ & $0.556\pm0.035$ & $0.556\pm0.027$ & $0.794\pm0.064$ & $\mathbf{0.833\pm0.091}$ & $0.816\pm0.059$ & $0.822\pm0.063$ \\
20 & $0.528\pm0.057$ & $0.583\pm0.023$ & $0.583\pm0.037$ & $0.583\pm0.027$ & $0.583\pm0.026$ & $0.842\pm0.050$ & $\mathbf{0.883\pm0.039}$ & $0.877\pm0.035$ & $0.881\pm0.039$ \\
50 & $0.653\pm0.044$ & $0.756\pm0.037$ & $0.844\pm0.060$ & $0.878\pm0.072$ & $0.822\pm0.047$ & $0.890\pm0.044$ & $0.927\pm0.024$ & $0.932\pm0.025$ & $\mathbf{0.936\pm0.022}$ \\
250 & $0.898\pm0.024$ & $0.945\pm0.025$ & $0.949\pm0.016$ & $0.953\pm0.010$ & $0.957\pm0.012$ & $0.955\pm0.013$ & $0.962\pm0.008$ & $0.964\pm0.005$ & $\mathbf{0.966\pm0.005}$ \\
500 & $0.940\pm0.010$ & $0.948\pm0.010$ & $0.950\pm0.006$ & $0.949\pm0.005$ & $0.952\pm0.007$ & $0.950\pm0.005$ & $0.961\pm0.009$ & $0.963\pm0.005$ & $\mathbf{0.964\pm0.005}$ \\
1.000 & $0.959\pm0.005$ & $0.939\pm0.007$ & $0.945\pm0.008$ & $0.948\pm0.008$ & $0.949\pm0.009$ & $0.949\pm0.006$ & $0.961\pm0.007$ & $0.965\pm0.005$ & $\mathbf{0.966\pm0.006}$ \\
5.000 & $\mathbf{0.977\pm0.003}$ & $0.947\pm0.003$ & $0.956\pm0.002$ & $0.958\pm0.002$ & $0.962\pm0.002$ & $0.956\pm0.002$ & $0.969\pm0.001$ & $0.973\pm0.001$ & $0.974\pm0.004$ \\
\bottomrule
\end{tabular}}
\vspace*{-15pt}
\end{table}

In this second set of numerical experiments,
we compare three different ans\"atze
in combination with linear support vector machines (SVMs).
The na\"ive approach
uses the Euclidean representation of the images 
as basis for the SVM.
Inspired by~\cite{Kolouri2016},
the second approach makes use of
the plain R-CDT projections 
using a fixed number of angles.
Finally, 
the third approach utilizes
our \mNRCDT{} projections
over the same set of angles.
For all these methods,
a $10$-fold cross validation is performed.
This means that the dataset is partitioned into ten subsets,
of which one is successively used for training,
whereas the remaining nine are reserved for testing.
The results
for different class sizes
and numbers of angels
are summarized for the academic dataset in Table~\ref{tab:max_comparison_syn}
and for the Lin\-MNIST dataset in Table~\ref{tab:max_comparison_mnist}. 
We observe that our approach outperforms all others, 
especially in the small data regime and for few angles.
For large data sizes, 
all methods perform at nearly the same accuracy.

\section{Conclusion}
In this work, 
we proposed the novel 
max-normalized R-CDT 
for feature representation 
and proved linear separability of classes 
generated by affine transforms of given templates.
This was validated by numerical experiments
showing a significant increase in classification accuracy
over original R-CDT.
Potential future directions include
the control of perturbations
either in the templates
or in the transforms
as well as
a more in-depth numerical study
in various applications.

%
%
\bibliographystyle{splncs04}
\bibliography{literature}

\end{document}